\documentclass{article}

\newcommand {\theoremstyle} [1] { }
\newenvironment{proof}{{\noindent\it\underline{Proof}}:}{\hfill$\Box$}

\newtheorem{thm}{Theorem}[section]
 
 \theoremstyle{plain}
 \newtheorem{lem}[thm]{Lemma} 
  
 \newtheorem{cor}[thm]{Corollary}
 \theoremstyle{definition}

 \theoremstyle{remark}
 \newtheorem{rem}[thm]{Remark}
 \newtheorem{exa}[thm]{Example}

\usepackage {amssymb}
\usepackage[dvipsnames]{xcolor}

\def\R{\mathbb{R}}

\usepackage {amssymb}
\usepackage{amsmath,amsfonts, color}
\usepackage{graphicx}

\title{{\bf\Large Multiplicity of Periodic Solutions  for  Dynamic Li\'enard Equations with  Delay  and Singular $\varphi$-laplacian of Relativistic Type}}
\author{{\large P. Amster, M. P. Kuna  and D. P. Santos }\hspace{2mm}
{\bf\large}\vspace{1mm}\\
{\small Departamento de Matem\'atica. Facultad de Ciencias Exactas y Naturales.}\\ 
{\small Universidad de Buenos Aires \& IMAS-CONICET.}\\
{\small Ciudad Universitaria. Pabell\'on I,(1428), Buenos Aires, Argentina.}\\
{\small E-mails: pamster@dm.uba.ar -- mpkuna@dm.uba.ar -- dsantos@dm.uba.ar}
\vspace{3mm}
}

\date{}
\begin{document}
\maketitle

\begin{abstract}
  We study the existence and multiplicity of periodic solutions  for singular $\varphi$-laplacian Li\'enard-like equations with delay on time scales. We prove the existence of multiple solutions  using  topological methods based on the Leray-Schauder degree. A special case is the 
  $T$-periodic problem for the forced pendulum equation and the sunflower equation with relativistic effects.

\end{abstract}
 
 \medskip



 \medskip
 \noindent


\section{Introduction}
In this work, we study  the existence and multiplicity of $T$-periodic solutions $x:\mathbb{T} \rightarrow \mathbb{R}$ to the following equation with delay on time scales
\begin{equation}\label{eq1}
(\varphi(x^{\Delta}(t)))^{\Delta} +h(x(t))x^{\Delta}(t)+ g(x(t-r)) =p(t)  \ \ \ t\in \mathbb{T},
\end{equation}
where ${\mathbb{T}}$ is an arbitrary $T$-periodic nonempty closed subset of  $\mathbb{R}$ (\textit{time scale}), $\varphi:(-a,a)\rightarrow \mathbb{R}$ is an
increasing homeomorphism  with $0<a<+\infty$ such that $\varphi(0) = 0$, and  $h,g:\mathbb{R} \rightarrow \mathbb{R}$ are continuous  functions. Moreover, we assume that  $r\ge 0$ and $T>0$ are  real numbers, and  that $p(t+T)=p(t)$ is continuous in $\mathbb{T}$ with $\overline p:=\frac{1}{T}\int_0^T p(t)\Delta t=0$. 
 
The time scales theory was introduced in 1988, in the PhD thesis of  Stefan Hilger \cite{hilgerthe}, as an attempt to unify discrete and continuous calculus.
 The time scale $\mathbb R$ corresponds to the continuous case and, hence, yields
 results for ordinary differential equations. If the time scale is ${\mathbb{Z}}$, 
then the results apply to difference equations. However, the generality of the set $\mathbb T$ produces many different situations in which the time scales formalism is useful  in several applications. For example,   in the study of hybrid discrete-continuous dynamical systems, see \cite{bo1}.

The methods usually employed to explore the existence of periodic solutions for dynamic equations in time scales are: fixed point theory \cite{kau,lw}, Mawhin’s continuation theorem \cite{che9,lh}, lower and upper solutions \cite{ste,gu} and variational methods \cite{gu1}, \cite{ti}, \cite{yu}, among many other works. 
{Some of the above cited references correspond to the {semilinear} case, that is, $\varphi(x)=x$ and some others to the \textit{$p$-laplacian operator}, namely $\varphi_p(x):= |x|^{p-2}x$. 
However, the literature concerning singular $\varphi$-laplacian operators in time scales is more scarce.}
 A special case of  (\ref{eq1}) with $\mathbb T=\R$ is the forced pendulum equation with relativistic effects, namely,

\begin{equation}\label{int}
 \left(\frac{x{'}}{\sqrt{1-\frac{x{'^2}}{c^{2}}}}\right)^\prime +kx{'}(t)+ b\sin x =p(t),  \ \ \ t\in \mathbb{R},
\end{equation}
where $c >0$ is the speed of light in the vacuum, $k > 0$ is a possible viscous friction
coefficient and $p$ is a continuous and $T$-periodic forcing term with mean value zero. This equation has received much attention  by several authors, see e.g. \cite{bre,man11,torres}. In  particular  in  \cite{torres}, {employing the Schauder fixed point theorem, Torres proved the existence of at least one $T$-periodic solution, provided that $2cT\le 1$. 
{This result was later improved in \cite{torres2} and finally in 
\cite{bjm}, where  the sharper condition  
$cT < \sqrt{3}\pi$} was obtained.} 
In the recent paper \cite{cid}, a new improvement was obtained in terms of $k$ and $\|p\|_{L^1}$ and allows to obtain the uniform condition $cT\le 2\pi$. 



 \smallskip

In this work, we generalize several aspects of the results in \cite{bjm} and  \cite{torres, torres2}. On the one hand, our problem consist of dynamical Li\'enard-like equations on time scales;
on the other hand, the functions $f, g$ are  general and the equation may also include a delay. This implies that  the use of the Poincar\'e operator does not reduce the problem to a finite-dimensional one, and requires the use of  accurate topological methods such as  the Leray-Schauder degree. 
{Moreover, our main theorem is in fact a multiplicity result, which intuitively can be motivated as follows. If we observe for example problem (\ref{int}), 
it is clear that the 
periodicity implies that if $x$ is a $T$-periodic solution, then $x+2k\pi$ is also a 
$T$-periodic solution for all $k\in \mathbb Z$. Such  solutions are usually called in the literature \textit{geometrically equivalent}. However, if the term $kx'$ is replaced by $h(x)x'$ for some continuous function $h$ close to a constant, then the problem still admits infinitely many solutions, which may be geometrically distinct if $h$ is not a $2\pi$-periodic function. 
With this idea in mind, it shall be shown that if the nonlinear term has a more general oscillatory behaviour, then  multiple solutions exist.  }


\medskip

More specifically, our main result reads as follows:

\begin{thm}
Assume that there exists  a strictly increasing sequence $\left\{\alpha_{j}\right\}_{j=0}^n$ 
such that 
$$
    (-1)^j\int_0^T h(x(t))x^{\Delta}(t)+g(x(t))\Delta t < 0   \hbox{ if  } \; x(0)= \alpha_{j}, \;    \left\|x^{\Delta}\right\|_{\infty} < a. 
$$
for every $j$ 
and each $C^1$ and $T$-periodic function   $x(t)$.
 {Then, for any continuous $T$-periodic function $p(t)$ with mean value zero, problem (\ref{eq1})  has at least 
  $n$ different  $T$-periodic solutions.}
  \end {thm}  

In particular, if $g$ 
is oscillatory over $\mathbb R$ and $h$ is locally monotone or locally close to a constant,
then (\ref{eq1})  has infinitely many different  $T$-periodic solutions, provided that the oscillations are sufficiently slow. More precisely, the following corollaries are obtained:

\begin{cor}
Assume that   there exists  a strictly increasing sequence $\left\{\alpha_{j}\right\}_{j=0}^n$ such that
\begin{center}
$(-1)^j g>0$ and $(-1)^jh$  is nonincreasing over $(\alpha_j-\frac {aT}2,\alpha_{j}+\frac {aT}2)$.    
\end{center}         
 Then, for any continuous $T$-periodic function $p(t)$ with mean value zero, problem (\ref{eq3}) has at least  $n$ different $T$-periodic solutions. 
\end{cor}

\begin{cor}
Assume  there exists  a strictly increasing sequence $\left\{\alpha_{j}\right\}_{j=0}^n$ and constants $\gamma_j$ such that
 
 \begin{center}
    $a|h(x)-\gamma_j|<(-1)^jg(x)$  
     for all $x\in (\alpha_j-\frac {aT}2,\alpha_{j}+\frac {aT}2)$. 
 \end{center}

 Then, for any continuous $T$-periodic function $p(t)$ with mean value zero, problem (\ref{eq3})  has at least $n$ different $T$-periodic solutions.    
\end{cor}

The proof of the preceding results  shall be based on the search for fixed points of an appropriate compact operator defined  on the Banach space of all continuous $T$-periodic functions  on $\mathbb{T}$. 
The singular nature of $\varphi$ will be of help in the obtention of the required a priori bounds, thus making possible a Leray-Schauder degree approach. 
{We highlight  that, in contrast with the continuous case, the  
treatment of Li\'enard-like equations on time scales is more delicate because the average of the 
term $h(x(t))x^\Delta(t)$ with $T$-periodic $x$ is not necessarily equal to $0$. This is due to the fact that the standard chain rule does not hold and, consequently,  extra conditions are required in order to avoid this difficulty.  }  

\medskip

The paper is organized as follows. In Section 2, we set the notation,
terminology, and  several preliminary results which will be used throughout this paper.  In Section $3$,  we adapt Mawhin's continuation theorem to the context of times scales in order to prove  the existence of at least one  $T$-periodic solution of (\ref{eq1}). 
In Section 4, we prove our main theorem 
{with the help of the arguments introduced in the preceding section}. Some examples illustrating the results are presented in Section 5.



\section{Notation and preliminaries}
\label{S:-1}

 For fixed $T>0$, we shall assume that $\mathbb T$ is $T$-periodic, that is, $\mathbb T + T=\mathbb T$. Moreover, since the equation includes a delay $r\ge 0$, we shall also assume that $\mathbb T -r \subset \mathbb T$. When $r>0$, it is observed that
 if $\mathbb T\ne \mathbb R$, then $r$ is necessarily commensurable with $T$, that is, $r=qT$ for some positive $q\in \mathbb Q$. Indeed, this is due to the fact that, otherwise, the set $\{e^{-2\pi \frac rT ni}\}_{n\in\mathbb N}$ is dense in $S^1\subset \mathbb C$ and the conclusion follows from the fact that $\mathbb T$ is closed.
 Note, also, that if $r$ is congruent to $\hat r$ modulo $T$, then $x(t-r)=x(t-\hat r)$ for any $T$-periodic function $x$ and, thus, we may assume without loss of generality that $r<T$.  
 For convenience, we shall also assume that
 $0\in \mathbb{T}$. 
 
Let us denote by  $C_{T}=C_{T}\left(\mathbb{T},\mathbb{R} \right)$ the Banach space of 
 all  continuous $T$-periodic functions  on $\mathbb{T}$  endowed with the uniform norm 
$\left\| x \right\|_{\infty}=\displaystyle  \sup _{\mathbb{T}} |x(t)| = \sup _{\left[0,T\right]_{\mathbb{T}}} |x(t)|$ 
and 
the closed subspace
 \begin{center}
 $\tilde{C_{T}}=\left\{x\in C_{T}:\int_{0}^{T} x(s)\Delta s=0\right\} $.
 \end{center}
For an element $x\in C_T$ its maximum and minimum values shall be denoted respectively by $x_M$ and $x_m$. 

Moreover, denote by $C^{1}_{T}=C^{1}_{T}\left(\mathbb{T},\mathbb{R} \right)$  the Banach space  of all continuous $T$-periodic functions on $\mathbb{T}$ that are $\Delta$-differentiable with continuous
$\Delta$-derivatives, endowed with the usual norm
 $$\left\|x\right\|_{1}= \sup _{\left[0,T\right]_{\mathbb{T}}} |x(t)|+ \displaystyle \sup _{\left[0,T\right]_{\mathbb{T}}}\left|x^{\Delta}(t)\right|.$$

We introduce the following operators and functions:

 \begin{itemize}
     \item 
The  \textit{Nemytskii operator} $N_f:C^{1}_{T} \rightarrow {C_{T}} $, given by
\begin{center}
$N_f (z)(t)=f(t,x(t),x^{\Delta}(t), x(t-r))$, 
\end{center}
where  $f:\mathbb{T}\times \mathbb{R}^3  \rightarrow \mathbb{R}$  is a continuous function;

\item The  \textit{integration operator}   $H:\tilde{C_{T}} \rightarrow  C^{1}_{T}$, 
\begin{center}
$ H(z)(t)=\int_{0}^t z(s)\Delta s$, 
\end{center}
\noindent
 \item The  continuous linear projectors:
\begin{center}
$Q:C_{T}\rightarrow C_{T}, \ \  Q(x)(t)=\frac{1}{T}\int_{0}^{T} x(s)\Delta s$,
\end{center}
\begin{center}
$P:C_{T} \rightarrow C_{T}, \ \  P(x)(t)=x(0)$ 
\end{center}
where, for convenience, the isomorphism between $\mathbb R$ and the subspace of constant functions of $C_T$ is omitted. 
 \end{itemize}

\medskip

The above equation  (\ref{eq1}) can be written as follows:
\begin{equation}\label{eq3}
(\varphi(x^{\Delta}(t)))^{\Delta} = f(t,x(t),x^{\Delta}(t), x({t-r})),  \ \ \ t\in \mathbb{T},
\end{equation}

A function $x\in C^{1}_{T}$ is said to be a solution of (\ref{eq3}) if $\varphi(x^{\Delta})$  is  of class $ C^{1}$ and  verifies
$(\varphi(x^{\Delta}(t)))^{\Delta} = f(t,x(t),x^{\Delta}(t), x({t-r}))$ for all $t\in \mathbb{T}$.

\medskip

The following  lemma is an adaptation of a result of \cite{ma1}  to time scales. 
\begin{lem}\label{lem1}
 For  each $x\in C_{T}$,  there exists a unique  $Q_{\varphi}=Q_{\varphi}(x)\in \left[x_{m}, x_{M} \right]$  such that
$$\int_{0}^{T} \varphi^{-1}(x(t)-Q_{\varphi}(x))\Delta t = 0.$$
Moreover, the function    $Q_{\varphi}:C_{T} \rightarrow \mathbb{R}$ is continuous  and sends bounded sets into bounded sets.
\end{lem}

\begin{proof} 
Let $x\in C_{T}$ and define the continuous application $G_{x}: \left[x_{m}, x_{M}\right] \rightarrow \mathbb{R}$  by
\begin{center}
$G_{x}(s)=\int_{0}^{T}\varphi^{-1}(x(t)-s)\Delta t$.
\end{center}
We claim that the equation
\begin{equation}\label{eq4}
G_{x}(s)=0
\end{equation}
has a unique solution  $Q_{\varphi}(x)$. Indeed, 
Let  $r,  s \in \left[x_{m}, x_{M}\right] $ be such that
\begin{center}
$\int_{0}^{T} \varphi^{-1}(x(t)-r)\Delta t=0 =\int_{0}^{T} \varphi^{-1}(x(t)-s)\Delta t$,
\end{center}
then using the fact  that  $\varphi^{-1}$ is strictly increasing we deduce that  $r=s$. 
Moreover,  
It is seen that
\begin{center}
$\int_{0}^{T} \varphi^{-1}(x(t)-x_M)\Delta t\le 0 \le \int_{0}^{T} \varphi^{-1}(x(t)-x_m)\Delta t$,
\end{center}
whence
\begin{center}
$G_{x}(x_{m})G_{x}(x_{M})\leq0$.
\end{center}
Thus, there exists  $s\in\left[x_{m}, x_{M}\right]$ such that  $G_{x}(s)=0$, that is,  equation (\ref{eq4}) has a unique solution. It follows that 
function $Q_{\varphi}: C_{T} \rightarrow \mathbb{R} $ 
given by $Q_\varphi(x)=s$ is well defined and, furthermore, 
because $s\in [x_m,x_M]$ we deduce that   
\begin{center}
$\left|Q_{\varphi}(x)\right|\leq \left\|x\right\|_{\infty}$.
\end{center}
Therefore, the function  $Q_{\varphi}$  sends bounded sets into bounded sets.

Finally, let us verify that  $Q_{\varphi}$ is continuous on $C_{T}$. Let $\left(x_{n}\right)_{n}\subset C_{T}$  be a sequence such 
that $x_{n}\rightarrow x$ in $C_{T}$. 
Since the function  $Q_{\varphi}$  sends bounded sets into bounded sets, the sequence   $\left(Q_{\varphi}(x_{n})\right)_{n}$ is bounded in $\mathbb R$ and, consequently, 
without loss of generality we may  assume that it converges to some $\tilde a$. 
Because 
\begin{center}
$  \int_{0}^{T} \varphi^{-1}(x_{n}(t)-Q_{\varphi}(x_{n}))\Delta t=0 $
\end{center}
for all $n$, by the  dominated convergence theorem on time scales \cite{bo1}, we deduce that
\begin{center}
$ \int_{0}^{T} \varphi^{-1}(x(t)-\widetilde{a})\Delta t=0$,
\end{center}
so  $Q_{\varphi}(h)=\widetilde{a}$. Thus, we conclude that the function  $Q_{\varphi}$ is continuous.
\end{proof}

\medskip
Now, we define a fixed point operator, which is similar to the one employed 
in  \cite{ma1} (see also  \cite{pablo} for an elementary introduction). In order to transform problem (\ref{eq3}) into a fixed point problem we  use  the operators $H,Q, N_{f}, P$ and  Lemma \ref{lem1}. The proof of this result is similar to the continuous case and shall not repeated here.

\begin{lem}\label{lem2}
$x \in C^{1}_{T}$ is a solution of  (\ref{eq3}) if and only if  $x$ is a fixed point of the operator   $M_{f}$ defined  on $C^{1}_{T}$ by
\begin{center}
$M_{f}(x)=P(x)  + Q(N_{f}(x)) + H\left( \varphi^{-1} \left[H(N_{f}(x)- Q(N_{f}(x)))-Q_{\varphi}(H(N_{f}(x)-Q(N_{f}(x))))\right] \right)$.
\end{center}
\end {lem}
 As the function   $f$ 
 is continuous, using  the Arzel\`a-Ascoli theorem it is not difficult to see that $M_f$ is completely continuous.

Using Lemma \ref{lem2}, the existence of a $T$-periodic solution for (\ref{eq3}) is reduced to the obtention of  fixed points of the operator $M_{f}$. To this end, we will use topological degree theory.

Consider the following family of problems defined for
$\lambda \in \left[0,1\right]$:
\begin{equation}\label{eq5}
 (\varphi(x^{\Delta}(t)))^{\Delta}=\lambda N_{f}(x)(t)+(1-\lambda)Q(N_{f}(x)),
\end{equation}
 where the operator  $N_{f}$ is defined by
 \begin{center}
$N_{f}(x)(t)= f(t,x(t),x^{\Delta}(t), x{(t-r)}):=-h(x(t))x^{\Delta}(t)-g(x(t-r)) + p(t),  \ \ \ t\in \mathbb{T}$. 
\end{center}
 
 For  each 
 $ \lambda \in [0,1]$, consider the  nonlinear operator   $M(\lambda,\cdot)$, where $M$   is defined on $[0,1]\times C^{1}_{T}$  by 
 
\begin{align}\label{eq6}
 M(\lambda,x)&=P(x)+Q(N_{f}(x)) + \\
 &H(\varphi^{-1}\left[\lambda H( N_f (x)-Q(N_{f}(x)))-  Q_{\varphi}( \lambda H(N_{f}(x)-Q(N_{f}(x)))) \right]).\nonumber
\end{align}
 
 Observe  that $M(1, x)=M_{f}$; moreover, similarly as above, it is easy
to see that  $M$ is completely continuous and that, for $\lambda>0$, 
the existence of solution to equation  (\ref{eq5}) is equivalent to the problem
\begin{center}
$x=M(\lambda, x)$.
\end{center} 
We claim that the previous assertion is true also for $\lambda=0$. Indeed, because $Q_\varphi(c)= c$ for any constant $c$, it is clear that 
$M(0,x)=P(x) + Q(N_f(x))$. 
If $x=M(0,x)$ then $x$ is constant and $x=P(x)$, that is, $Q(N_f(x))=0$ and 
(\ref{eq5}) with $\lambda =0$ is trivially satisfied. Conversely, if 
$ (\varphi(x^{\Delta}(t)))^{\Delta}\equiv Q(N_{f}(x))$ then we obtain, upon integration,
$\int_0^T Q(N_{f}(x))\Delta t=0$ which, in turn, implies that $Q(N_{f}(x))=0$. 
Thus $x^\Delta$ is constant and, by periodicity,
$x^\Delta\equiv 0$, that is, $x$ is constant and, consequently, $x=P(x)=P(x) + Q(N_{f}(x))=M(0,x)$.

\begin{rem}\label{re1}
It is worthy to notice that, for any $\lambda \in [0,1]$,  if $x$ is a fixed point of $M$  then   $Q(N_{f}(x))= 0$.
\end{rem}

 
 \section{Continuation theorem}

In this section, we establish the continuation theorem that shall be employed for 
the proof of our main result. Let us denote by $deg_{B}$ and  $deg_{LS}$ the Brouwer and 
Leray-Schauder degrees respectively. 
The following result is obtained as in the continuous case; 
we include a proof for the sake of completeness. 

\begin{thm}\label{teo1}
Assume that  $\Omega$ is an open bounded set in  $C^{1}_{T}$ such that the following conditions hold:
\begin{enumerate}
\item For each $ \lambda \in (0,1)$ the problem
\begin{equation}\label{eq7}
(\varphi(x^{\Delta}(t)))^{\Delta} = \lambda N_{f}(x)
\end{equation}
has no solution on $\partial\Omega$.
\item  The equation
\begin{center}
$g(y)=0$,
\end{center}
has no solution on $\partial\Omega \cap \mathbb{R}$, where we consider the natural identification  of  $\mathbb R$ with the subspace of constant functions of $C^{1}_{T}$.
\item  The Brouwer degree of $g$ satisfies: 	
\begin{center}
$deg_{B}(g,\Omega \cap  \mathbb{R},0)\neq 0$.
\end{center}
\end{enumerate}
Then 
problem (\ref{eq1}) has at least one $T$-periodic  solution. 
\end {thm}

\begin{proof} 
  Let  $ \lambda \in (0,1]$. If $x$ is  a solution of (\ref{eq7}), then $Q(N_{f}(x))=0$, hence  $x$ is a solution of problem (\ref{eq5}). On the other hand, for $ \lambda \in (0,1]$, let  $x$ be a solution  of (\ref{eq5}) and since
\begin{center}
$ Q\left(\lambda N_{f}(x)+(1-\lambda)Q(N_{f}(x))\right)=Q(N_{f}(x))$,
\end{center}
it follows that $Q(N_{f}(x))=0$, whence  $x$ is a solution  of (\ref{eq7}). It is deduced that, for  $ \lambda \in (0,1]$, problems  (\ref{eq5}) and  (\ref{eq7}) have the same solutions. We assume that  (\ref{eq5}) has no  solutions on $\partial\Omega$ for  $\lambda=1$, since otherwise we are done with the proof. It follows that (\ref{eq5}) has no solutions  for  $ (\lambda, x) \in  (0,1]\times \partial\Omega$. 
If $x$ is a solution of (\ref{eq5}) for 
$\lambda=0$,  then  
we conclude as before that $Q(N_f(x))=0$ and 
$x(t)\equiv b\in \mathbb R$. 
Thus, 
 using the fact that $\int_0^T p(t)\Delta t=0$ 
\begin{center}
$0= \displaystyle \frac{1}{T}\int_0^T f(t,b,0,b)\Delta t=-g(b)$,
\end{center}
which, together with hypothesis 2, implies that $b \notin \partial\Omega$. 

Summarizing,  we proved that  (\ref{eq5}) has no solution on  $\partial\Omega$ for all  $ \lambda \in \left[0, 1\right]$. Thus, for each  $ \lambda \in \left[0, 1\right]$, the  Leray-Schauder degree  $deg_{LS}(I-M(\lambda, \cdot), \Omega, 0)$ is well defined and, by the homotopy invariance  property, 
\begin{center}
$deg_{LS}(I-M(0,\cdot),\Omega,0)= deg_{LS}(I-M(1,\cdot),\Omega,0)$.
\end{center}
On the other hand, 
\begin{center}
$deg_{LS}(I-M(0,\cdot),\Omega,0) = deg_{LS}(I-(P+QN_{f}) ,\Omega,0)$.
\end{center}
But the range of the mapping  
\begin{center}
$z \mapsto P(z)+QN_{f}(z)$
\end{center}
is contained in the subspace  of constant functions of $C^{1}_{T}$, identified with $\mathbb{R}$. Thus, using the reduction property  of the Leray-Schauder degree \cite{man7, man12}
\begin{align*}
 deg_{LS}(I-(P+QN_{f}) ,\Omega,0) &= deg_{B}\left(I-(P+QN_{f})\left|_{\overline{\Omega \cap  \mathbb{R}}}\right. ,\Omega\cap \mathbb{R}, 0\right)\\
&=deg_{B}(g,\Omega\cap  \mathbb{R},0)\neq 0.
\end{align*}
Then, $deg_{LS}(I-M(1,\cdot),\Omega,0)\neq 0$ and,
in consequence,  there exists  $x\in\Omega$ such that  $M_{f}(x)=M(1,x)=x$, which is a 
solution  of (\ref{eq3}) and therefore  a solution  of (\ref{eq1}).
\end{proof} 

\medskip

With the help of Theorem \ref{teo1} we shall be able to prove the existence of fixed points of $M_{f}$
With this aim, for $\lambda\in (0,1]$ we consider the equation 
\begin{equation}\label{eq10}
\left(\varphi\left(x^{\Delta}(t)\right)\right)^{\Delta} + \lambda h(x(t))x^{\Delta}(t)+ \lambda g(x(t-r)) =\lambda p(t)  \ \ \ t\in \mathbb{T},
\end{equation}
which is the explicit expression of problem (\ref{eq7}).

{The  next example shows that the  $ \int_{0}^{T}h(x(t))x^{\Delta}(t)\Delta t$  is not necessarily equal to zero.   This is due to the fact that the standard chain rule does  not hold for time scales.}
{
 \begin{exa}\label{ej5}
  Let $\mathbb T$ be $3$-periodic with   $[0,3]_\mathbb{T} = [0,1] \cup  \left\{ 2, 3\right\}$, let $h(x)=x$, and let $x:\mathbb{T} \rightarrow \mathbb{R}$ be the $3$-periodic function defined on $[0,3)_\mathbb T$ by
 $$ x(t)= \left\{ \begin{array}{lcc}
             t &   if  & 0\leq t\leq 1 \\
             \\ 2 &  if & t=2. 
             \end{array}
   \right.$$
  It follows by direct computation that $
  \int_{0}^{3}x(t)x^{\Delta}(t)\Delta t=-\frac 52$. 
  \end{exa}}
   
\medskip

\begin{lem}\label{monot-cond}
Assume that $h$ is nondecreasing (resp. nonincreasing) over the range of $x\in C^1_T$. Then
$$\int_0^T h(x(t))x^\Delta(t)\Delta t \le 0 \quad (\hbox{resp.} \ge 0).
$$

\end{lem}

\begin{proof}
Consider the primitive of $h$, namely 
\begin{center}
$\mathcal H(x)=\int_{0}^{x}h(s)ds$
\end{center}
and observe that $\mathcal H\circ x$ is $\Delta$-differentiable. 
Moreover, if $t_0$ is right dense, then  $(\mathcal H \circ x)^\Delta(t_0)= h(x(t_0))x^\Delta(t_0)$. On the other hand, if $t_0$ is right-scattered, then
$$
(\mathcal H \circ x)^\Delta(t_0)= h(\xi) x^\Delta(t_0)
$$
for some $\xi$ between $x(t_0)$ and $x(\sigma(t_0))$. If $h$ is nondecreasing over the range of $x$, it readily follows that 
$$(\mathcal H \circ x)^\Delta(t_0) \ge  h(x(t_0)) x^\Delta(t_0)
$$
and, consequently,
$$
\int_0^T h(x(t))x^\Delta(t)\Delta t \le  \int_0^T (\mathcal H\circ x)^\Delta(t)\Delta t =0. 
$$
The opposite inequality is obtained if we assume, instead, that $h$ is nonincreasing.  
\end{proof}

 
 \section{Multiplicity of periodic solutions}

In this section
we establish the existence of at least $n$ different solutions of problem (\ref{eq1}).  The statements in the introduction are repeated here, for the sake of clarity. 

\begin{thm}\label{teo4}
Assume that there exists  a strictly increasing sequence $\left\{\alpha_{j}\right\}_{j=0}^n$ 
such that for all $j$ and 
$x\in C_T^1$,
\begin{equation}
    \label{condicion1}
    (-1)^j\int_0^T [h(x(t))x^{\Delta}(t)+g(x(t))]\Delta t < 0   \hbox{ if  } \; x(0)= \alpha_{j}, \;    \left\|x^{\Delta}\right\|_{\infty} < a. 
\end{equation}
 Then, for any continuous $T$-periodic function $p(t)$ with mean value zero, problem (\ref{eq1})  has at least 
   $n$ different   
 $T$-periodic solutions.
\end {thm}  

\begin{proof} 
Assume that 
 $x\in C^{1}_{T}\left(\mathbb{T},\mathbb{R} \right)$ is  a solution of (\ref{eq10}) with $ \lambda \in (0,1]$, then  $ | x^{\Delta}(t)| < a$   and
 $$\int_0^T [h(x(t))x^{\Delta}(t)+g(x(t-\tau))]\Delta t =0.
 $$
From the periodicity of $x$ we deduce from (\ref{condicion1}) that
 $x(0)\neq \alpha_j$, for any $j=0,\ldots, n$. Moreover, 
 (\ref{condicion1}) for $x\equiv \alpha_j$ also implies that 
$(-1)^jg(\alpha_j)<0$.
 Therefore, problem (\ref{eq5})  has  no  solution  in   $\partial\Omega_j$ for all $j=0,\ldots, n-1$, where 
\begin{center}
$\Omega_j:=\left\{x\in C^{1}_{T}\left(\mathbb{T},\mathbb{R}\right)/ x(0)\in (\alpha_j, \alpha_{j+1}), \   \left\|x^{\Delta}\right\|_{\infty} < a \right\}$.
\end{center}

From the homotopy invariance of the Leray-Schauder degree, we obtain
\begin{align*}
 deg_{LS}(I-M(1,\cdot),\Omega_j,0)&= deg_{LS}(I-M(0,\cdot),\Omega_j,0)=\\
& = deg_{LS}(I-(P+QN_{f}) ,\Omega_j,0)=\\
& = deg_{B}\left(I-(P+QN_{f})\left|_{\overline{\Omega_j \cap  \mathbb{R}}}\right. ,\Omega_j\cap \mathbb{R}, 0\right)=\\
&= deg_{B}(g,\Omega_j\cap  \mathbb{R},0)=\\
&= deg_{B}(g,(\alpha_{j}, \alpha_{j+1}),0) =(-1)^j.
\end{align*}
We conclude that the operator $M(1,\cdot)=M_{f}$ has a fixed point $x_j\in \Omega_j$.  
Finally, observe that   
$x_j(0)\in (\alpha_j, \alpha_{j+1})$ 
hence all the solutions are different. 
\end{proof}


\begin{rem}
 It is clear that the sign in condition (\ref{condicion1}) may be reversed, that is: 
 $$(-1)^{j}\int_0^T [h(x(t))x^{\Delta}(t)+g(x(t))]\Delta t > 0   \hbox{ if  } \; x(0)= \alpha_{j}, \;    \left\|x^{\Delta}\right\|_{\infty} < a. $$
\end{rem}

The next corollary shows that condition (\ref{condicion1}) can be obtained from 
  appropriate  explicit assumptions on $g$ and $h$. 

\begin{cor}\label{monot}
Assume that   there exists  a strictly increasing sequence $\left\{\alpha_{j}\right\}_{j=0}^n$ such that
\begin{center}
$(-1)^j g>0$ and $(-1)^jh$  is nonincreasing over $(\alpha_j-\frac {aT}2,\alpha_{j}+\frac {aT}2)$.    
\end{center}         
 Then, for any continuous $T$-periodic function $p(t)$ with mean value zero, problem (\ref{eq3}) has at least  $n$ different $T$-periodic solutions. 
\end{cor}
{\begin{proof}
 From the previous proof and Lemma \ref{monot-cond} it suffices to verify that
if $x\in C^{1}_{T}$ satisfies $x(0)=\alpha_j$ and $\|x^\Delta\|_\infty<a$, then 
$x(t)\in (\alpha_j-\frac {aT}2,\alpha_{j}+\frac {aT}2)$ for all $t$. 
To this end, observe that if $|x(t)- \alpha_j| \ge \frac {aT}2$ for some $t\in (0,T)_\mathbb T$, then 
$$
\frac {aT}2\le |x(t)- \alpha_j| \le \int_0^t |x^\Delta(s)|\Delta s < at,
$$
whence $t>\frac T2$. Due to the periodicity, we also deduce that $T-t > \frac T2$, a contradiction. 
\end{proof}
}

\begin{rem}\label{longitud} In particular, the conditions in the previous theorem imply that
$\alpha_{j+1}- \alpha_{j}\ge aT $    for $j=0,1, \ldots, n-1$.
\end{rem}

The alternative condition that $h$ is locally close to a constant is obtained in the following corollary: 

\begin{cor}
\label{cor4}
Assume  there exists  a strictly increasing sequence $\left\{\alpha_{j}\right\}_{j=0}^n$ and constants $\gamma_j$ such that

 \begin{center}
    $a|h(x)-\gamma_j|<(-1)^jg(x)$  
     for all $x\in (\alpha_j-\frac {aT}2,\alpha_{j}+\frac {aT}2)$. 
 \end{center}

 Then, for any continuous $T$-periodic function $p(t)$ with mean value zero, problem (\ref{eq3})  has at least $n$ different $T$-periodic solutions.    
\end{cor}
 \begin{proof}
Suppose for example that $x$ is a solution of (\ref{eq10}) with $j$ even and
$x(0)=\alpha_j$. Because 
  $\|x^\Delta\|_\infty < a$ and $x(t)\in (\alpha_j-\frac {aT}2,\alpha_{j}+\frac {aT}2)$, then 
  $$(h(x(t))-\gamma_j) x^\Delta(t) + g(x(t)) \ge 
  g(x(t)) - a|h(x(t))-\gamma_j | >0.$$ 
Since $\int_0^T\gamma_j x^\Delta(t)\Delta t=0$, we deduce that  
$$\int_0^T [h(x(t))x^\Delta(t)  + 
g(x(t-\tau))]\Delta t=\int_0^T [h(x(t))x^\Delta(t)  + 
g(x(t))]\Delta t >0.
$$
An analogous reasoning for $j$ odd shows that condition (\ref{condicion1}) is fulfilled. 

 \end{proof}
 
\begin{rem}\label{infinitely}
In particular, 
suppose that $g$ has slow oscillations, that is,  there exists a sequence of zeros $x_j\nearrow +\infty$ such that 
$(-1)^jg(x) > 0$ for $x\in (x_j,x_{j+1})$, with $x_{j+1}-x_j> aT$, then the problem has infinitely many solutions, provided  that $(-1)^jh$ is nonincreasing or
$a|h(x)|<|g(x)|$ in $(\alpha_j-\frac {aT}2, \alpha_j+\frac {aT}2)$ for all $j$, where $\alpha_j=\frac {x_j+x_{j+1}}2$.

\end{rem}

 \section{Examples} 
 
 In order to illustrate the above results, we consider some examples.
 \begin{exa}\label{eje1}
  Let us consider the equation
 \begin{equation}\label{eq21}
\left(\frac{x^{\Delta}(t)}{\sqrt{1-x^{\Delta }(t)^2}}\right)^{\Delta} + e^{-x^{2}(t)}x^{\Delta}(t)+  \arctan (x(t))= \sin(4\pi t)  \ \ \ t\in \mathbb{T}\ \ \ 
\end{equation}
 where $\mathbb T$ is a $1/2$-periodic time scale with 
 $$\left[0,1/2\right]_{\mathbb{T}}= \left[0,1/8\right] \cup  \left\{3/16\right\}  \cup \left\{1/4\right\} \cup  \left[5/16,3/8\right]  \cup \left[7/16,1/2\right].$$ 
 By Corollary \ref{monot} or Corollary \ref{cor4} with $\alpha_0 \ll 0\ll \alpha_1$
 we deduce that  (\ref{eq21}) has at least one $1/2$-periodic solution. 
\end{exa}

\begin{exa}\label{eje2} 
 Let $h:\mathbb{R} \rightarrow \mathbb{R}$ be continuous. Let us study the existence of a $2\pi$-periodic solution to the
following problem
 \begin{equation}\label{eq22}
 \left(\frac{x^\Delta(t)}{\sqrt{1-\frac{x^{\Delta }(t)^2}{c^{2}}}}\right)^\Delta + h(x(t))x^\Delta(t)+ x^{3}(t-r)= \cos(t),  \ \ \ t\in \mathbb{R},
\end{equation}
 where $c>0$ and $r\geq 0$. Using Corollaries \ref{monot} and \ref{cor4}, it follows that problem (\ref{eq22}) has at least one $2\pi$-periodic solution if  one  of the following assumptions is verified:
 {\begin{enumerate}
     
     \item There exists $R>0$ such that
$$h(y) \le h(x) \qquad \hbox{for}\,\, y\ge x\ge R \, \hbox{ or }\,  y\le x\le -R.
$$
     \item  $\limsup_{x\to \pm \infty} \left|\frac {h(x)}{x^3}\right| < 1$.
\end{enumerate}}
\end{exa}


 {
  \begin{exa}\label{eje4}
  Let us consider the relativistic pendulum equation on time scales
 \begin{equation}\label{eq25}
 \left(
 \frac {x^\Delta(t)}
 {\sqrt {1- \frac{x^\Delta(t)^2}{c^{2}} }}\right)^\Delta + h(x(t))x^\Delta(t)+ \sin(x(t)) =p(t) \ \ \ t\in \mathbb{T},
\end{equation}
where $h, p:\mathbb{R} \rightarrow \mathbb{R}$ are  continuous functions and $p$ is $T$-periodic  with mean value zero. 
If $cT\le \pi$, then problem (\ref{eq25}) has  infinitely many $T$-periodic  solutions under one of the following assumptions: 
\begin{enumerate}
    
\item 
 $(-1)^{j}h$ is nonincreasing in 
$(\alpha_j- \frac {cT}2, \alpha_j +\frac {cT}2)$, where $\alpha_j= (2j+1)\frac{\pi}{2}$ for $j\in \mathbb Z$. 
\item $c\left|h(x)+\gamma_j\right| < \left|\sin(x)\right|$ for 
$x\in (\alpha_j- \frac {cT}2, \alpha_j +\frac {cT}2)$ for $j\in \mathbb Z$ and some constants $\gamma_j$.  
\end{enumerate}
Clearly, both conditions are satisfied when $h$ is constant although, in this case, 
 the solutions are not
 necessarily different in geometric sense
  (see \cite{torres}). 
  {It is worth observing that the
 restriction $cT \le \pi$, which comes from Remark \ref{longitud}, improves the one in the original work by Torres, but it is slightly worse 
 than the one obtained in \cite{torres2} which, as mentioned in the introduction, reads $cT<2\sqrt 3=3.46\ldots$ However, the method in \cite{torres2} involves a change of variables that cannot be extended to a general time scale. The sharper bound  given in \cite{bjm} is easily obtained in the continuous case, due to the Sobolev inequality $$\|x-\overline x\|_\infty ^2 \le \frac T{12}\|x'\|_{L^2} ^2,
$$
which holds for $T$-periodic functions.
Indeed, it suffices to observe that, if we replace $P$ by $Q$ in the definition of the operator $M$ in (\ref{eq6}) then
our main theorem is also valid,  changing $x(0)$ by $\overline x$ in condition 
(\ref{condicion1}) and the definition of $\Omega$. Thus, any possible solution of (\ref{eq5}) satisfying for example $\overline x=\frac \pi 2$ verifies $|x(t)-\frac \pi 2| \le \frac{cT}{2\sqrt 3}$ for all $t$. If $cT\le \sqrt 3\pi$, then $x(t)\in [0,\pi]$ for all $t$ and 
$$0 = \int_0^T \sin(x(t))\, dt >0,
$$
a contradiction. For a general time scale, the argument is essentially the same and yields the condition $s(\mathbb T)c\sqrt T \le \frac \pi 2$, where $s(\mathbb T)$ is the constant of the corresponding Sobolev inequality. We recall that, in the continuous case, the obtention of the value $s(\mathbb R)=\frac T{12}$ relies 
on the Fourier series expansion for periodic functions  
(see e.g. \cite{maw}), which should be adapted accordingly to the general context. For example, a rapid computation shows, for arbitrary $\mathbb T$, that $s(\mathbb T)\le \frac T4$ which, applied to  this case, retrieves the condition $cT\le \pi$. 
} 
 \end{exa}}

\section*{Acknowledgements} 
This research was partially supported by  CONICET/Argentina and project
UBACyT 20020160100002B.

\bibliographystyle{plain}

\end{document}